\documentclass[12pt,british]{article}

\usepackage{babel}
\usepackage{mathrsfs,diagbox} 
\usepackage{amsmath,amsfonts,amssymb,amsthm}
\usepackage{ucs} 
\usepackage[utf8x]{inputenc} 
\usepackage[T1]{fontenc} 
\usepackage{enumerate}
\usepackage{textcomp,url}
\usepackage{eucal}
\usepackage{comment}
\usepackage[noBBpl]{mathpazo}
\usepackage[matrix,arrow]{xy}
\usepackage{epsfig}
\usepackage[pdftex,                %
implicit=true,
bookmarks=true,
breaklinks=true,
bookmarksnumbered = true,
colorlinks= true,
urlcolor= green,
anchorcolor = yellow,
citecolor=blue,
pdfborder= {0 0 0}
]{hyperref}

\usepackage{ytableau} 
\usepackage{multicol}

\makeatletter

\renewcommand{\p@enumii}{}
\def\@enum@{\list{\csname label\@enumctr\endcsname}%
{\usecounter{\@enumctr}\def\makelabel##1{
\normalfont\ignorespaces\emph{{##1}~}}
\setlength{\labelsep}{3pt}
\setlength{\parsep}{0pt}
\setlength{\itemsep}{0pt}
\setlength{\leftmargin}{0pt}
\setlength{\labelwidth}{0pt}
\setlength{\listparindent}{\parindent}
\setlength{\itemsep}{0pt}
\setlength{\itemindent}{0pt}
\topsep=3pt plus 1pt minus 1 pt}}
\makeatother

\renewcommand{\epsilon}{\ensuremath{\varepsilon}}
\renewcommand{\phi}{\ensuremath{\varphi}}

\renewcommand{\to}{\ensuremath{\longrightarrow}}

\newcommand{\C}{\ensuremath{\mathbb C}}

\newcommand{\N}{\ensuremath{\mathbb N}}
\newcommand{\Z}{\ensuremath{\mathbb Z}}

\newcommand{\sn}[1][n]{\ensuremath{S_{{#1}}}}

\newcommand{\largeleft}{\mbox{$\left(\raisebox{-4mm}{}\right.$}}
\newcommand{\largeright}{\mbox{$\left.\raisebox{-4mm}{}\right)$}}

\renewcommand{\ker}[1]{\ensuremath{\operatorname{\text{Ker}}\left({#1}\right)}}
\newcommand{\im}[1]{\ensuremath{\operatorname{\text{Im}}\left({#1}\right)}}
\newcommand{\aut}[1]{\ensuremath{\operatorname{\text{Aut}}\left({#1}\right)}}
\newcommand{\autc}[1]{\ensuremath{\operatorname{\text{Aut}_c}\left({#1}\right)}}

\newcommand{\barpn}[1][n]{\ensuremath{\overline{P}_{#1}}}

\makeatletter
\def\@map#1#2[#3]{\mbox{$#1 \colon\thinspace #2 \to #3$}}
\def\map#1#2{\@ifnextchar [{\@map{#1}{#2}}{\@map{#1}{#2}[#2]}}
\makeatother

\newcommand{\brak}[1]{\ensuremath{\left\{ #1 \right\}}}
\newcommand{\ang}[1]{\ensuremath{\left\langle #1\right\rangle}}
\newcommand{\set}[2]{\ensuremath{\left\{#1 \,\mid\, #2\right\}}}

\newcommand{\setangl}[2]{\ensuremath{\ang{\left. #1 \,\right\rvert \, #2}}}
\newcommand{\ord}[1]{\ensuremath{\left\lvert #1\right\rvert}}

\newcommand{\setl}[2]{\ensuremath{\brak{\left. #1 \,\right\rvert \, #2}}}

\newcommand{\sign}{{\rm sign}}

\newtheoremstyle{theoremm}{}{}{\itshape}{}{\scshape}{.}{ }{}
\theoremstyle{theoremm}

\newtheorem{thm}{Theorem}
\newtheorem{lem}[thm]{Lemma}
\newtheorem{prop}[thm]{Proposition}

\newtheoremstyle{remark}{}{}{}{}{\scshape}{.}{ }{}
\theoremstyle{remark}

\newtheorem{rem}[thm]{Remark}

\newtheoremstyle{comment}{}{}{\bfseries}{}{\bfseries}{:}{ }{}
\theoremstyle{comment}

\newcommand{\reth}[1]{Theorem~\protect\ref{th:#1}}
\newcommand{\relem}[1]{Lemma~\protect\ref{lem:#1}}

\allowdisplaybreaks

\begin{document}

\title{The R$_\infty$ property for pure Artin braid groups}

\author{KAREL~DEKIMPE\\
KU Leuven Campus Kulak Kortrijk,\\
Etienne Sabbelaan 53, 8500 Kortrijk, Belgium.\\
e-mail:~\texttt{karel.dekimpe@kuleuven.be}\vspace*{4mm}\\
DACIBERG~LIMA~GON\c{C}ALVES\\
Departamento de Matem\'atica - IME-USP,\\
Rua~do~Mat\~ao~1010~CEP:~05508-090  - S\~ao Paulo - SP - Brazil.\\
e-mail:~\texttt{dlgoncal@ime.usp.br}\vspace*{4mm}\\
OSCAR~OCAMPO~\\
Universidade Federal da Bahia,\\
Departamento de Matem\'atica - IME,\\
Av. Adhemar de Barros~S/N~CEP:~40170-110 - Salvador - BA - Brazil.\\
e-mail:~\texttt{oscaro@ufba.br}
}

\date{\today}

\maketitle

\begin{abstract}
In this paper we prove that all pure Artin braid groups $P_n$ ($n\geq 3$) have the $R_\infty$ property. In order to obtain  this result, we analyse  the naturally induced morphism $\aut{P_n} \to  \aut{\Gamma_2 (P_n)/\Gamma_3(P_n)}$ which turns out to factor through a representation $\rho\colon \sn[n+1] \to \aut{\Gamma_2 (P_n)/\Gamma_3(P_n)}$. We can then use representation theory of the symmetric groups to show that any automorphism $\alpha$ of $P_n$ acts on the free abelian group $\Gamma_2 (P_n)/\Gamma_3(P_n)$ via a matrix with an eigenvalue equal to 1. This allows us to conclude that the Reidemeister number $R(\alpha)$ of $\alpha$ is $\infty$.
 \end{abstract}

	\let\thefootnote\relax

\section{Introduction}

Consider a group $G$ and an endomorphism $\alpha$ of $G$. We say that two elements $x$ and $y$ of $G$ are twisted conjugate (via $\alpha$) if and only if there exists a $z\in G$ such that $x = z y \alpha(z)^{-1}$. It is easy to see that the relation of being twisted conjugate is an equivalence relation and the number of equivalence classes (also referred to as Reidemeister classes) is called the Reidemeister number $R(\alpha)$ of $\alpha$. This Reidemeister number is either a positive integer or $\infty$.

Reidemeister numbers find their origins in algebraic topology and to be more precise in Nielsen--Reidemeister fixed point theory. Here one is interested in the number of fixed point classes of a selfmap $f$ of a space $X$. This number is called the Reidemeister number $R(f)$ of the map $f$, and one can show that $R(f)=R(f_\ast)$, where $f_\ast\colon \pi_1(X) \to \pi_1(X) $ is the induced endomorphism on the fundamental group $\pi_1(X)$ of $X$.

There is currently a growing interest in the study of groups $G$  that have the $R_\infty$ property, these are groups for which $R(\alpha)=\infty$ for any automorphism $\alpha \in \aut{G}$. The study of groups with that  property was initiated by Fel'shtyn and Hill \cite{FH}.

Since the beginning of this century many authors have been studying this property  and for several families of groups it is known whether or not they have the $R_\infty$ property. To list some examples of groups with the $R_\infty$ property, we mention the non-elementary Gromov hyperbolic groups \cite{F,LL}, most of the Baumslag--Solitar groups \cite{FG1}  and groups quasi--isometric to Baumslag--Solitar groups \cite{TW},\
 generalized Baumslag--Solitar groups \cite{L}, many linear groups \cite{FN,N}, several families of lamplighter groups \cite{GW,T} ,  \ldots

In \cite{FG} it was shown that the Artin braid groups $B_n$ also have property $R_\infty$. The goal of this paper is to show that the same holds for the pure braid groups $P_n$. The technique that we will use is to look at a suitable quotient of $P_n$. Indeed, if $\alpha$ is an automorphism of a group $G$ and $N$ is a normal subgroup of $N$ with $\alpha(N)=N$ (e.g.\ when $N$ is a characteristic subgroup of $G$) then $\alpha$ induces an automorphism $\bar\alpha$ of $G/N$. It is easy to see that $R(\alpha) > R(\bar\alpha)$, so if 
$R(\bar\alpha)= \infty$, then also $R(\alpha)=\infty$. In fact, for $n\geq 5$ we will show that for any automorphism $\alpha$ of $P_n$, the induced automorphism on $P_n/\Gamma_3(P_n)$ has an infinite Reidemeister number.  
Here  $\Gamma_i(P_n)$ stands for the $i$-th term of the lower central series (see the following section for the exact definition of the terms of the lower central series). To be more precise, we will show that the induced automorphism on the subquotient $\Gamma_2(P_n)/\Gamma_3(P_n)$ has infinite Reidemeister number  , and this will also allow us to conclude that $R(\alpha)=\infty$. 

In the next section we start by recalling some basic facts about Artin braid groups and we describe in detail the induced automorphisms on $\Gamma_2(P_n)/\Gamma_3(P_n)$ coming from automorphisms of $P_n$. It turns out that these form a finite group of automorphisms which is isomorphic to $\sn[n+1]$, the symmetric group on $n+1$ letters. In the last section, we then exploit this fact, using results from representation theory of $\sn$, to prove the $R_\infty$ property for $P_n$ (with $n\geq 5$). The cases $P_3$ and $P_4$ are treated separately. 

\subsection*{Acknowledgments}
The first author was supported by long term structural funding – Methusalem grant of the Flemish Government.
The second author was partially supported by  the Projeto Tem\'atico-FAPESP Topologia Alg\'ebrica, Geom\'etrica e Diferencial 2016/24707-4 (Brazil).
The third author  was partially supported by Capes/Programa Capes-PrInt/ Processo n\'umero 88881.309857/2018-01.

\section{Pure braid groups and their automorphisms}\label{sec-autos}

Let us start by very briefly recalling some facts about Artin braid groups (see~\cite{Ha} for more details). It is well known that the Artin braid group $B_{n}$ can be generated by the so-called elementary braids $\sigma_{1},\ldots,\sigma_{n-1}$ that are subject to the following relations:
\begin{equation}\label{eq:artin1}
\left\{ \begin{gathered}
\text{$\sigma_{i} \sigma_{j} = \sigma_{j}  \sigma_{i}$ for all $1\leq i<j\leq n-1$ such that $\ord{i-j}\geq 2$}\\
\text{$\sigma_{i+1} \sigma_{i} \sigma_{i+1} =\sigma_{i} \sigma_{i+1} \sigma_{i}$ for all $1\leq i\leq n-2$.}
\end{gathered}\right.
\end{equation}
The homomorphism $\map{\sigma}{B_{n}}[\sn]$ mapping $\sigma_{i}$ to the transposition $(i,i+1)$ is surjective and the kernel of $\sigma$ is the pure Artin braid group $P_{n}$ on $n$ strings. So there is a short exact sequence:
\begin{equation}\label{eq:sespn}
1 \to P_n \to  B_n \stackrel{\sigma}{\to} \sn \to 1.
\end{equation}
There is a set of generators $\brak{A_{i,j}}_{1\leq i<j\leq n}$ of $P_n$, where:
\begin{equation}\label{eq:defaij}
A_{i,j}=\sigma_{j-1}\cdots \sigma_{i+1}\sigma_{i}^{2} \sigma_{i+1}^{-1}\cdots \sigma_{j-1}^{-1}.
\end{equation}
For notational reasons it will be handy to set $A_{j,i}= A_{i,j}$ for $1\leq i < j \leq n$.

\medskip

Now, let $G$ be any group. For $g,h\in G$ we define the commutator of $g$ and $h$ as $[g,h]=ghg^{-1}h^{-1}$ and if $H,K$ are subgroups of $G$ then  $[H,K]=\setangl{[h,k]}{h\in H,\, k\in K}$. The terms of the \textit{lower central series} $\brak{\Gamma_i(G)}_{i\in \N}$ of $G$ are defined iteratively by  letting $\Gamma_1(G)=G$ and $\Gamma_{i+1}(G)=[G,\Gamma_i(G)]$ for all  $i\in \N$. The groups $\Gamma_i(G)$ are characteristic subgroups of $G$.

\medskip

It is known that the consecutive quotients $\Gamma_i(P_n)/\Gamma_{i+1}(P_n)$ are free Abelian groups of finite rank for all $i$ (see e.g.\ \cite{FR}).
In particular $P_n/\Gamma_2(P_n)$ is isomorphic to $\Z^{n(n-1)/2}$, and a basis of $P_n/\Gamma_2(P_n)$ is given by the images of the generators $\brak{A_{i,j}}_{1\leq i<j\leq n}$ in the quotient $P_n/\Gamma_2(P_n)$.

\medskip

To prove the $R_\infty$ property for the groups $P_n$ we have to study the automorphims of $P_n$. We follow the description of $\aut{P_n}$ that can be found in a paper by  Bardakov, Neshchadim and Singh \cite{BNS}. Throughout the rest of this section we assume that $n\geq 4$. 

It is known that the center $Z(P_n)$ of $P_n$ is an infinite cyclic subgroup of $P_n$ and that  $P_n \cong Z(P_n)\times \barpn$, where $\barpn=P_n/Z(P_n)$. The center $Z(P_n)$ is generated by the full twist braid
$A_{1,2} A_{1,3} A_{2,3} \cdots A_{1,n} A_{2,n} \cdots A_{n-1,n}$ and is denoted by $z_n$ in \cite{BNS}.
Since $Z(P_n)$ is characteristic in $P_n$, there is a homomorphism $\phi\colon \aut{P_n}\to \aut{\barpn}$ and a whole series of homomorphisms $\phi_i\colon \aut{P_n}\to \aut{\frac{\Gamma_{i}(\barpn)}{\Gamma_{i+1}(\barpn)}}$ for all $i\geq 1$.

\begin{rem}\label{technique}
Note that if $\alpha\in \aut{P_n}$ and $R(\phi(\alpha))=\infty$ or $R(\phi_i(\alpha))=\infty$ for some $i$, then $R(\alpha)=\infty$. Indeed, if $R(\phi_i(\alpha))=\infty$ then also for the induced automorphism $\bar\alpha \in\aut{\frac{\overline P_n}{\Gamma_{i+1}(\overline P_n)}}= \aut{\frac{P_n}{Z(P_n) \Gamma_{i+1}(P_n)}}$ it holds that $R(\bar \alpha)= \infty$ (\cite[Lemma 2.2]{DG}) and hence $R(\alpha) =\infty$.\\
 We will use this technique for $i=2$ to prove the $R_\infty$ property for $P_n$, for $n\geq 4$.
\end{rem}

An automorphism $\alpha \in \aut{P_n}$ is called central if $\alpha$ induces the identity on $\barpn$ (so if $\phi(\alpha)=1$). The subgroup of $\aut{P_n}$ of all central automorphisms is denoted by $\autc{P_n}$.
It is known that $\aut{P_n} \cong \autc{P_n} \rtimes \aut{\barpn}$ and so for our purposes, it suffices to understand $\phi(\aut{P_n})=\phi(\aut{\barpn})$.

In \cite[page 6]{BNS} an explicit list of automorphisms  $\omega_1, \omega_2, \ldots, \omega_n$ and $\epsilon$ of $\aut{P_n}$ is given and it is shown that the $\phi(\omega_i)$ ($i=1,\ldots , n$) together with  $\phi(\epsilon)$ are generators of $\aut{\barpn}$. Let  $\overline{\omega}_i=\phi_1(\omega_i)$ and $\overline{\epsilon}=\phi_1(\epsilon)$. So $\overline{\omega}_i$, $\overline{\epsilon}$ are the induced morphisms on $\frac{\Gamma_{1}(\barpn)}{\Gamma_{2}(\barpn)}$. 

First of all note that $\overline{\epsilon}(\bar{A}_{i,j})=-\bar{A}_{i,j}$ for $1\leq i<j\leq n$,  where we use $\bar{A}_{i,j}$ to denote the images of the $A_{i,j}$ under the natural projection of $P_n$ to  $\frac{\Gamma_{1}(\barpn)}{\Gamma_{2}(\barpn)}$. This follows easily from the formula of $\epsilon$ on page~6 of \cite{BNS}. 

In \cite[Equation~(2.1)]{BNS} we find the relations (without bars)
\begin{align*}
\overline{\omega}_i\overline{\omega}_j=\overline{\omega}_j\overline{\omega}_i \textrm{ for } \left.i-j\right. > 1\\
\overline{\omega}_i\overline{\omega}_{i+1}\overline{\omega}_i=\overline{\omega}_{i+1}\overline{\omega}_i\overline{\omega}_{i+1}\\
\overline{\omega}_1\cdots \overline{\omega}_{n-1}\overline{\omega}_n^2\overline{\omega}_{n-1}\cdots \overline{\omega}_1=1\\
(\overline{\omega}_1\overline{\omega}_2\cdots \overline{\omega}_n)^{n+1}=1\\
(\overline{\epsilon}\overline{\omega}_i)^2=1 \textrm{ and } \overline{\epsilon}^2=1.
\end{align*}
Of course, since $\overline{\epsilon}$ is equal to minus the identity on $\frac{\Gamma_{1}(\barpn)}{\Gamma_{2}(\barpn)}$ we also have that $\overline{\epsilon}\overline{\omega}_i=\overline{\omega}_i\overline{\epsilon}$. 
It then follows easily that we have the relations
\begin{align*}
\overline{\omega}_i\overline{\omega}_j=\overline{\omega}_j\overline{\omega}_i \textrm{ for } \left.i-j\right. > 1\\
\overline{\omega}_i\overline{\omega}_{i+1}\overline{\omega}_i=\overline{\omega}_{i+1}\overline{\omega}_i\overline{\omega}_{i+1}\\
\overline{\omega}_i^2=1\\
\overline{\epsilon}\overline{\omega}_i=\overline{\omega}_i\overline{\epsilon}=1 \\ 
\overline{\epsilon}^2=1.
\end{align*}
and hence
\begin{equation*}
\phi_1(\aut{P_n})=\ang{\overline{\omega}_1, \overline{\omega}_2, \ldots, \overline{\omega}_n, \overline{\epsilon}}\cong \sn[n+1]\times \Z_2
\end{equation*}
where $\sn[n+1]=\ang{\overline{\omega}_1, \overline{\omega}_2, \ldots, \overline{\omega}_n}$ and $\Z_2=\ang{\overline{\epsilon}}$.

At this point we want to remark that there are indeed no extra relations because otherwise $\ang{\overline{\omega}_1, \overline{\omega}_2, \ldots, \overline{\omega}_n, \overline{\epsilon}}$ would be a quotient of $\sn[n+1]\times \Z_2$, which could only be $\Z_2\times \Z_2$ or $\Z_2$ (since $n+1\geq 5$) and using the explicit formula for the $\overline{\omega}_i$  that we will give in a moment, it is easy to see that this is not the case. Indeed, let $k\in \{1,2,\ldots,n-1\}$ and let $\tau_k=(k,\,k+1)\in \sn$ be a transposition, then the formulas on page~6 of \cite{BNS} imply that 
\begin{equation}\label{images:phi1}
\overline{\omega}_k(\bar A_{i,j}) = \bar A_{\tau_k(i), \tau_k(j)} \mbox{  \ \ for $1 \leq k < n$.} 
\end{equation}
The description of $\overline{\omega}_n$ is a little bit more complicated, but actually we will not need it in this paper.

By induction on $i$, it is easy to see that $\ker{\phi_1}\subseteq \ker{\phi_i}$ for all $i=2,3,\ldots$ (in fact the induced automorphism $\phi_i(\alpha)$ is completely determined by $\varphi_1(\alpha)$ as we will see explicitly for  $i=2$) and so 
it follows that for all $\phi_i\colon \aut{P_n}\to \aut{\frac{\Gamma_{i}(\barpn)}{\Gamma_{i+1}(\barpn)}}$ it holds that $\im{\phi_i}$ is a finite group.

\medskip

Now let us describe $\phi_2$ in more detail. The group $\Gamma_2(P_n)/\Gamma_3(P_n)=
\Gamma_2(\overline P_n)/\Gamma_3(\overline P_n)$ is free abelian of rank $\displaystyle {n \choose 3}$ and
in \cite[Page~170]{GGO2} it was proven that the set
\begin{equation}\label{eq:basis1}
\mathcal{B}=\setl{\alpha_{i,j,k}}{1\leq i<j<k\leq n},
\end{equation}
where $\alpha_{i,j,k}$ is the natural projection of $[A_{i,j}, A_{j,k}]$ into $\frac{\Gamma_2(\overline P_n)}{\Gamma_3(\overline P_n)}$ is a basis of the free abelian group $\frac{\Gamma_2(\overline P_n)}{\Gamma_3(\overline P_n)}=\frac{\Gamma_2(P_n)}{\Gamma_3( P_n)}$. Moreover, also the following relations hold for all $1\leq i < j< k \leq n$:
\begin{equation}\label{relations:alpha}
\overline{ [A_{i,k}, A_{j,k}]}  = \alpha_{i,j,k}^{-1},\;
\overline{[A_{i,j}, A_{i,k}]}  = \alpha_{i,j,k}^{-1},\mbox{ and }
\overline{[A_{j,k}, A_{i,j}]}  = \alpha_{i,j,k}^{-1},
\end{equation}
where we now use a bar to indicated the natural projection of an element to $\frac{\Gamma_2(\overline P_n)}{\Gamma_3(\overline P_n)}$.

\begin{lem} \label{action:generators}
For $k=1, 2, \ldots , n-1$ we let $\tau_k$ denote the transposition $(k,\,k+1)\in \sn$. We then have 
for all $1\leq r <s <t \leq n$: 
\begin{enumerate}
\item $\phi_2(\epsilon)( \alpha_{r,s,t}) = \alpha_{r,s,t}$\vspace{2mm}
\item $\phi_2(\omega_k)( \alpha_{r,s,t}) = \left\{ \begin{array}{ll}
\alpha_{\tau_k(r),\tau_k(s),\tau_k(t)} & \mbox{if } \tau_k(r) < \tau_k(s) < \tau_k(t)\\
\alpha^{-1}_{\tau_k(s),\tau_k(r),\tau_k(t)} & \mbox{if } \tau_k(s)< \tau_k(r) \\
\alpha^{-1}_{\tau_k(r),\tau_k(t),\tau_k(s)} & \mbox{if } \tau_k(t)< \tau_k(s) \\
\end{array} \
\right. $
\end{enumerate} 
\end{lem}
\begin{proof}
Recall that $\epsilon(A_{i,j}) = A_{i,j}^{-1} \gamma_{i,j}$ for some $\gamma_{i,j} \in Z(P_n)\cdot\Gamma_2(P_n)$. Hence 
\[ \phi_2(\epsilon) ( \alpha_{r,s,t}) = \overline{[ \epsilon(A_{r,s}) ,\epsilon( A_{s,t})]} =\overline{[ A_{r,s}^{-1} \gamma_{r,s}, A_{s,t}^{-1} \gamma_{s,t} ]} = 
 \overline{[ A_{r,s}^{-1} , A_{s,t}^{-1} ]} = \overline{[ A_{r,s} , A_{s,t}]}  \]
 showing that $\phi_2(\epsilon)$ is the identity map.
  
 From \eqref{images:phi1} we get that $\omega_k(A_{i,j}) = A_{\tau_k(i),\tau_k(j)} \delta_{i,j}$ for some $\delta_{i,j} \in 
  Z(P_n)\cdot \Gamma_2(P_n)$ and so
  \[ \phi_2(\omega_k) (\alpha_{r,s,t}) = \overline{[\omega_k(A_{r,s}), \omega_k(A_{s,t}) ]}= 
  \overline{ [ A_{\tau_k(r),\tau_k(s)},A_{\tau_k(s),\tau_k(t)}] }.
  \]
  The lemma now follows using the relations \eqref{relations:alpha}
\end{proof}
Note that since $\epsilon$ induces the identity map on $\frac{\Gamma_{2}(\barpn)}{\Gamma_{3}(\barpn)}$, it folllows that  $\im{\phi_2}\cong \sn[n+1]$. Under this isomorphism, the elements $\phi_2(\omega_k)$  ($1 \leq k \leq n$) correspond with the transpositions $\tau_k=(k, k+1)$. Generalizing the lemma above, we are able to give 
a very good  description of how the subgroup $\langle \phi_2(\omega_1), \phi_2(\omega_2), \ldots, \phi_2(\omega_{n-1})\rangle \cong \sn$ of $\im{\phi_2}$ acts on $\frac{\Gamma_2(\overline P_n)}{\Gamma_3(\overline P_n)}$.

\begin{prop}\label{prop:action}
Let $n\geq 4$. Consider an automorphism  $\alpha\in \aut{P_n}$ for which $\varphi_2(\alpha) \in \langle \phi_2(\omega_1), \phi_2(\omega_2), \ldots, \phi_2(\omega_{n-1})\rangle\cong \sn$. Let $\pi \in \sn$ be the permutation corresponding to $\phi_2(\alpha)$.   
Then
\begin{equation}
\phi_2(\alpha)( \alpha_{r,s,t} ) = 
\begin{cases}
\alpha_{\pi(r),\pi(s),\pi(t)} & \text{if $\pi(r)< \pi(s)< \pi(t)$}\\
\alpha_{\pi(s),\pi(t),\pi(r)} & \text{if $\pi(s)< \pi(t)< \pi(r)$}\\
\alpha_{\pi(t),\pi(r),\pi(s)} & \text{if $\pi(t)< \pi(r)< \pi(s)$}\\
\alpha^{-1}_{\pi(r),\pi(t),\pi(s)} & \text{if $\pi(r)< \pi(t)< \pi(s)$}\\
\alpha^{-1}_{\pi(s),\pi(r),\pi(t)} & \text{if $\pi(s)< \pi(r)< \pi(t)$}\\
\alpha^{-1}_{\pi(t),\pi(s),\pi(r)} & \text{if $\pi(t)< \pi(s)< \pi(r)$}\\
\end{cases}.
\end{equation}
\end{prop}

\begin{proof}
In Lemma~\ref{action:generators} we already proved this result in case $\alpha$ is equal to one of the 
generators $\omega_k$ (and where the corresponding permutation is  the transposition $\tau_k$). We can use the same technique now to prove the general proposition. We can assume that $\alpha \in \langle \omega_1, \omega_2, \ldots, \omega_{n-1} \rangle$ and by applying \eqref{images:phi1} inductively we find that $\alpha(A_{i,j}) = A_{\pi(i),\pi(j)} \gamma_{i,j}$ for some $\gamma_{i,j}\in Z(P_n) \cdot \Gamma_2(P_n)$. Just as in the proof of Lemma~\ref{action:generators}
we then find that 
\[ \phi_2(\alpha) ( \alpha_{r,s,t} ) = \overline{ [A_{\pi(r), \pi(s)}, A_{\pi(s), \pi(t)} ]}. \]
The result now follows by considering the six possibilities to order the three numbers $\brak{\pi(r), \pi(s), \pi(t)}$ and by using the relations \eqref{relations:alpha}.
\end{proof}

\section{The pure Artin braid group has the $R_{\infty}$ property}

In this section, we will prove the main result of our paper, i.e.~any pure braid group $P_n$ with $n\geq 3$ has the $R_\infty$ property.

\subsection{Low dimensional cases}

We have that $P_3=F_2\oplus \Z$, where $F_2$ is the free group on two generators and $\Z= Z(P_3)$. Hence $\overline{P}_3=F_2$ has the $R_{\infty}$ property, from which it follows that $P_3$ has the $R_\infty$ property.  

\begin{lem}\label{lem:p4}
Let $\alpha\in \aut{\Z^k}$, where $k\geq 1$. 
If $1$ or $-1$ is an eigenvalue of $\alpha$, then $R(\alpha)\geq 2$. 
\end{lem}

\begin{proof}
It is well known that $R(\alpha)=+\infty$ if $1$ is an eigenvalue of $\alpha$, so in this case certainly $R(\alpha)\geq 2$. So, from now onwards we assume that $1$ is not an eigenvalue of $\alpha$. In this case 
\begin{equation*}
R(\alpha) = \ord{\det{(I_k-A)}}
\end{equation*}
with $I_k$ the $k\times k$-identity matrix and $A$ is the matrix representing $\alpha\in \aut{\Z^k}\cong GL_k(\Z)$ with respect to a chosen free generating set $e_1,e_2,\ldots,e_k$ of $\Z^k$. 

\medskip

Let $-1$ be an eigenvalue of $\alpha$ and take $V=\set{z\in \Z^k}{\alpha(z)=-z}\neq \brak{0}$. 
Then $\frac{\Z^k}{V}$ is torsion free. Indeed assume that $\overline{z}$ is a torsion element in $\frac{\Z^k}{V}$, where $\overline{z}$ denotes the natural projection of an element $z\in \Z^k$ in $\frac{\Z^k}{V}$. Then there exists $k>0$ such that $k\overline{z}=\overline{0}$, or such that $kz\in V$. But this means that $\alpha(kz)=-kz$ and hence $k\alpha(z)=-kz$ which implies that $\alpha(z)=-z$ or $z\in V$ and so $\overline{z}=\overline{0}$. 
Hence $V\cong \Z^{\ell}$, with $\ell>0$, and $\frac{\Z^{k}}{V}\cong \Z^m$ for $0\leq m<k$. 

It follows that we can choose a free generating set $e_1,e_2,\ldots, e_{\ell},e_{\ell+1},\ldots,e_k$ of $\Z^k$ such that $V=\ang{e_1,e_2,\ldots, e_{\ell}}$. With respect to this generating set, $\alpha$ has a matrix representation 
\begin{equation*}
A=\left(
\begin{array}
[c]{cc}%
-I_{\ell} & \star \\
0 & A' 
\end{array}
\right)
\end{equation*}
with $I_{\ell}$ the $\ell\times \ell$-identity matrix and $A'$ a $(k-\ell)\times (k-\ell)$ integral matrix. Then 
\begin{equation*}
\begin{array}{rcl}
R(\alpha) & = & \ord{\det( I_k - A )}\\
 & = & \ord{\det{(2I_{\ell})}} \, \ord{ \det{ (I_k-A') } } \\
& = & 2^{\ell}\cdot \ord{ \det(I_k-A') }.\\
\end{array}
\end{equation*}
Since $1$ is not an eigenvalue of $A$ (and so also not of $A'$),  $0<\ord{ \det(I_k-A') }\in \Z$, therefore $R(\alpha)\geq 2^{\ell}$, with $\ell>0$.
\end{proof}

The following proposition seems to be well known to the experts. However, since we were unable to find a reference for this result, we present a proof here.

\begin{prop}\label{productformula}
Let $N$ be a finitely generated nilpotent group and $\alpha$ an endomorphism of $N$. Suppose that 
\[ 1=N_0 \subseteq N_1 \subseteq N_2  \subseteq  \cdots \subseteq N_{i-1} \subseteq N_{i} \subseteq \cdots \subseteq N_c=N\] 
is a central series with $\alpha(N_i)\subseteq N_i$ and $N_{i}/N_{i-1}$ torsion free for all $i=1,2,\ldots, c$.
Let $\alpha_i\colon N_i/N_{i-1} \to N_i/N_{i-1}$ denote the induced endomorphism. Then 
\[ R(\alpha)= \prod_{i=1}^c R(\alpha_i).\]
\end{prop}

\begin{proof}

If $c=1$, there is nothing to show. So suppose $c>1$ and let $\overline\alpha\colon N/N_1 \to N/N_1$ denote the induced endomorphism. By induction, it suffices to show that $R(\alpha)=R(\alpha_1) R(\overline\alpha)$. If $R(\overline\alpha) = \infty$ then also $R(\alpha)=\infty$ and there is nothing to show. So suppose that $R(\overline\alpha)<\infty$.  

Now, choose $y_j\in N$, with $j\in J$, such that their naturals projections $\bar{y}_j$ are representatives for the different Reidemeister classes of $\overline{\alpha}$ (so $J$ is a finite set). Choose also representatives $x_i$ ($i\in I$) in $N_1$ for the different Reidemeister classes of $\alpha_1$. We claim that the elements $x_i y_j$ ($i\in I$, $j\in J$) uniquely represent all the Reidemeister classes of $\alpha$. 

First we show that any element $g\in N$ is Reidemeister equivalent to one of the $x_iy_j$. Certainly, there is a $j\in J$ such that $\bar{g}$ is equivalent to $\bar{y}_j$. This implies that there is an $h\in N$ and an $x\in N_1$ such that 
$g = x h y_j \alpha(h)^{-1}$. As $x\in N_1$, there is a $z\in N_1 \subseteq Z(N)$ and an $i\in I$ with $x= z x_i \alpha(z)^{-1}$. As $x_i,\; z$ and $\alpha(z) \in Z(N)$, we get that $g = zh x_iy_j \alpha(zh)^{-1}$, showing that 
$g$ is equivalent to $x_iy_j$. 

Now assume that $x_{i_1} y_{j_1}$ and $x_{i_2} y_{j_2}$ are equivalent. By looking at the projection to $N/N_1$, it follows that $j_1=j_2=j$, so we are in the situation that $x_{i_1} y_j $ is equivalent to $x_{i_2} y_j$ and we need to show that $i_1=i_2$. So there must exist a $g\in N$ with $x_{i_1} y_j = g x_{i_2} y_j \alpha(g)^{-1}$. 
If $g\in N_1$, then also $\alpha(g)\in N_1$ and this equation reduces to $x_{i_1} = g x_{i_2} \alpha_1(g)^{-1}$, from which it  follows that $i_1=i_2$. So assume that $g \not\in N_1$ and let $k>1$ be the smallest integer such that $g \in N_k$.  From $x_{i_1}y_j = g x_{i_2}y_j \alpha(g)^{-1}$ we get, using the fact that $x_{i_1}, x_{i_2}\in Z(N)$ that 
\begin{eqnarray*}
x_{i_1}x_{i_2}^{-1}  & = & y_j^{-1} g y_j \alpha(g)^{-1} \\
 & \Downarrow & \\
 x_{i_1}x_{i_2}^{-1}  N_{k-1} & = & y_j^{-1} g y_j g^{-1} g \alpha_k(g)^{-1} N_{k-1} \\
  & \Downarrow & \\
  1 N_{k-1} & =  & g \alpha_k(g)^{-1} N_{k-1}
\end{eqnarray*}
where we used that $y_j^{-1} g y_j g^{-1} = [y_j, g^{-1}] \in N_{k-1}$ since we are working with a central series.
Hence $g N_{k-1}$ is a fixed point of $\alpha_{k}$. However, by \cite[Lemma 2.2]{DG}, the fact that $R(\overline\alpha)< \infty$ implies that $\alpha_k$ does not have 1 as an eigenvalue (recall that $\alpha_k$ is an endomorphism of a free abelian group $N_k/N_{k-1}$ of finite rank, say $r$,  and so we can view $\alpha_k$ as an $r\times r$ integral matrix). This implies that Fix$(\alpha_k)=1$. Hence $gN_{k-1} = 1 N_{k-1}$, which implies that $g \in N_{k-1}$, which is a contradiction. 
\end{proof}

\begin{thm}
$P_4$ has the $R_{\infty}$ property.
\end{thm}

\begin{proof}
Since $P_4''=\left[ [P_4,P_4], [P_4,P_4] \right]= \overline{P}_4''$ is characteristic in $P_4$, it is enough to show that for every $\alpha\in \aut{P_4}$ the induced action of $\alpha$ on the group 
$G=\frac{\overline P_4}{\overline{P}_4''}$ has infinite Reidemeister number. By \cite[Theorem 1.1]{CS} we know that the groups $\Gamma_k(G)/\Gamma_{k+1}(G)$ are free abelian for all $k\geq 1$ and that $\Gamma_k(G)/\Gamma_{k+1}(G)\cong \Z^{5(k-1)}$ for all $k\geq 3$. 
Now let $\alpha\in \aut{P_4}$ be an automorphism then $\alpha$ induces a series of automorphisms $\alpha_k\in \aut{\Gamma_k(G)/\Gamma_{k+1}(G)}$ and  $\overline{\alpha}_c\in \aut{G/\Gamma_{c+1}(G)}$. 
Since the factors $\Gamma_k(G)/\Gamma_{k+1}(G)$ are torsion-free, it follows from the proposition above that $R(\overline{\alpha}_c)=\prod_{k=1}^c R(\alpha_k)$. 

For even $k\geq 3$ we have that $\operatorname{rank}\left( \Gamma_k(G)/\Gamma_{k+1}(G) \right)$ is odd, because it is equal to $5(k-1)$.  Moreover, $\alpha_k$ is of finite order since we already know that the induced maps $\phi_k\colon \aut{P_n} \to \aut{\Gamma_k(\overline P_n)/\Gamma_{k+1} (\overline P_n)}$ have finite image. Hence $\alpha_k$ has a real eigenvalue, which must be $1$ or $-1$. For any $c$ we have $R(\alpha)\geq R(\overline{\alpha}_c)$. 
Applying \relem{p4} we obtain a second inequality
\begin{equation*}
R(\overline{\alpha}_{2\ell})\geq R(\alpha_4)R(\alpha_6)\cdots R(\alpha_{2\ell})\geq 2^{\ell-1}
\end{equation*}
and so we must have that $R(\alpha)\geq 2^{\ell-1}$ for all $\ell$. Hence $R(\alpha)=\infty$. 
\end{proof}

\subsection{General cases}

To prove that the groups $P_n$ with $n\geq 5$ have the $R_\infty$ property we will show that for any 
automorphism $\alpha\in \aut{P_n}$ the induced automophism $\phi_2(\alpha) \in \aut{\frac{\Gamma_2(\overline P_n)}{\Gamma_3(\overline P_n)}}$ has eigenvalue 1 and so $R(\phi_2(\alpha))=\infty$, from which it follows (see Remark~\ref{technique}) that $R(\alpha)=\infty$.

Recall from Section \ref{sec-autos} that $\phi_2$ factors through $\im{\phi_1}/\langle \overline \epsilon \rangle$ and so there exists a representation 
\begin{equation}\label{def-representation}
\rho\colon \frac{\im{\phi_1}}{\langle \overline \epsilon \rangle} \cong S_{n+1} \to  \aut{\frac{\Gamma_2(\overline P_n)}{\Gamma_3(\overline P_n)}} 
\end{equation}
such that $\phi_2= \rho \circ p $, where $p$ is the composition: 
\[ \aut{P_n} \stackrel {\phi_1}{\to} \im{\phi_1}= 
\langle \overline \omega_1, \ldots \overline \omega_n, \overline \epsilon\rangle \to 
\frac{\im{\phi_1}}{\langle \overline \epsilon \rangle}= \langle \overline \omega_1, \ldots \overline \omega_n\rangle\cong \sn[n+1].\] 
We can assume that under the isomorphim 
$\frac{\im{\phi_1}}{\langle \overline \epsilon \rangle} \cong S_{n+1} $, the element $\overline{\omega}_i$ corresponds to the transposition $(i,\, i+1)$. 

We know that  $\frac{\Gamma_2(\overline P_n)}{\Gamma_3(\overline P_n)}\cong \Z^{ n \choose 3}$, so we can view $\rho$ as a complex representation $\rho\colon S_{n+1} \to {\rm GL}( {n \choose 3} , \Z) \subset {\rm GL}( {n\choose 3} , \C) $. As a first step, we will explicitely determine this representation.

For this, we recall some background on the representation theory of symmetric groups. 
The irreducible representations of $\sn$ are indexed by partitions of $n$ \cite{JK},
where a partition $\lambda$ of $n$ is a sequence $(\lambda_1 , \lambda_2, \ldots, \lambda_{\ell})$ of positive integers, with $\lambda_1 \geq \lambda_2 \geq \cdots \geq \lambda_\ell$ and $\lambda_1 + \lambda_2 + \cdots + \lambda_\ell = n$.

Associated with a partition $\lambda$ as above is its \emph{Young diagram}, which is formally defined as the set
\begin{equation*}
Y(\lambda)=\set{(i,j)\in \Z^2}{1\leq i\leq \ell, \,1\leq j\leq \lambda_i}
\end{equation*}
and which we can visualise as a set of cells in the plane as follows:
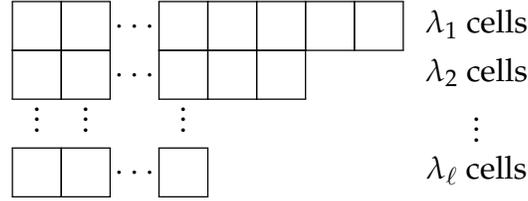
\[
\begin{ytableau}
\mbox{ }   & \mbox{ } & \none[\cdots]  & \mbox{ }&  \mbox{ }& \mbox{ }&  \mbox{ }& \mbox{ } & \none & \none[\makebox(0,8){$\lambda_1$ cells}] \\
\mbox{ }   & \mbox{ } & \none[\cdots]  & \mbox{ }&  \mbox{ }& \mbox{ } & \none& \none & \none & \none[\makebox(0,8){$\lambda_2$ cells}]  \\
 \none[\vdots]  & \none[\vdots] & \none &  \none[\vdots] &  \none &\none & \none & \none & \none & 
 \none[\makebox(0,8){$\vdots$}] \\
 \mbox{ } & \mbox{ } & \none[\cdots]  & \mbox{ }&  \none& \none & \none  & \none & \none & \none[\makebox(0,8){$\lambda_\ell$ cells}]  \\
\end{ytableau}
\]
In the above picture, the element $(i,j)\in Y(\lambda)$ corresponds to the cell in the $i$-th row and the $j$-column, as an explicit example, the Young diagram of the partition $\lambda=(5,2,2,1)$ of 10 is given by 
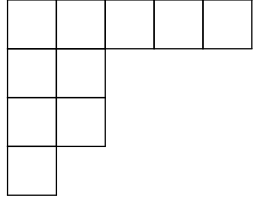
\[ 
\begin{ytableau}
\mbox{} & \mbox{} & \mbox{} & \mbox{} & \mbox{} \\
\mbox{} & \mbox{} & \none & \none & \none \\
\mbox{} & \mbox{} & \none & \none & \none \\
\mbox{} & \none & \none & \none & \none \\
\end{ytableau}
\] 

There are two irreducible 1-dimensional representation of $\sn$ for every $n$. The first one is the trivial representation and the second one is the sign representation $\sign\colon \sn \to \{1,-1\}$, mapping even permutations to 1 and odd permutations to $-1$. For any representation $\psi\colon \sn \to {\rm GL}(m,\C)$, we get a new one of the same dimension $\sign \cdot \psi\colon \sn \to  {\rm GL}(m,\C)$, with $(\sign \cdot \psi) (\mu) =
\sign(\mu) \cdot \psi(\mu)$. The following is a trivial, but useful remark as we will see in a moment:

\begin{rem}\label{rem:evenperm}
If $\chi_{\psi}$ and $\chi_{\sign \cdot \psi}$ denote the characters related to a representation $\psi$ of $\sn$, then for any even permutation $\mu$ we have $\chi_{\psi}(\mu)=\chi_{\sign\cdot \psi}(\mu)$. 
\end{rem}

\begin{thm}\label{rho:known}
Let $n\geq 6$. The representation $\rho\colon S_{n} \to {\rm GL} ( {n-1 \choose 3 }, \C)$ which is determined 
by  $\phi_2\colon \aut{P_{n-1}} \to \aut{ \frac{\Gamma_2(\overline P_{n-1})}{\Gamma_3(\overline P_{n-1})}} $ 
as in \eqref{def-representation} is irreducible and the associated Young diagram  has shape $(n-3,1,1,1 )$.
\end{thm}

\begin{proof}

We will divide the proof in two cases, namely the case for larger values of $n$ and the case for smaller values.

\medskip

We will first consider the situation in which $n\geq 13$ and at the end of the proof we explain how to deal with the remaining $n$ via an easy  case by case study, given some more details for $n=8$.

\medskip

In \cite[Result 3]{R}, for $n\geq 15$ a list of the first 7 minimal dimensions of irreducible representations of $S_n$ are given and are numbered as $(A)$, $(B)$, \ldots, $(F)$ and $(G)$. In this proof, we will only need the first 6 minimal dimensions and by checking explicitly the character tables of $S_{13}$ and $S_{14}$ (e.g.\ using GAP \cite{GAP}) one can see that also in these two cases the first six minimal dimensions are given by formulas $(A)$ to $(F)$. Of course, the first dimension $(A)$ is just 1 and one can see that the last dimension $(F)$ is ${n-1 \choose 3}$, the dimension of the representation $\rho$.  Each of these dimensions correspond to two representations say $\psi$ and $\sign \cdot \psi$.  
 
Let $\psi_A$ denote the trivial representation of $\sn$ and $\chi_A$ the trivial character. 
In  \cite[Table~5.1]{GG} we find the characters of one of the representations for each of the dimensions $(B)$ to $(F)$ (and $(G)$, which we don't need). Let us denote by $\psi_B, \ldots, \psi_F$ these representations and
let us denote by $\chi_B, \chi_C, \ldots, \chi_F$ the corresponding characters. We will use 
$\chi_{A'}, \chi_{B'}, \ldots, \chi_{F'}$ to denote the characters of the representations $\sign \cdot \psi_A$, 
$\sign \cdot \psi_B, \ldots, \sign \cdot \psi_F$.  
 
\medskip

Let $\mu \in \sn$. Then we can decompose $\mu$ into a disjoint cycle decomposition. We let $a_1$ denote the number of 1-cycles in this decomposition (or fixed points of $\mu$), $a_2$ the number of 2-cycles in this decomposition and $a_3$ the number of 3-cycles. In the table below, we summarise the information that we will 
use from \cite[Table~5.1]{GG}. (Note that $(A)$ is not present in  \cite[Table~5.1]{GG} and $(F)$ corresponds to the last line of the table and not to the second last line as one might expect). 

\begin{center}
\begingroup
\renewcommand{\arraystretch}{1.5}
\begin{tabular}{|c|c|c|}
\hline
Representation & Associated $\lambda$ & character evaluated in $\mu$ \\ \hline
$\psi_A$ & $ (n )$ & $ 1$ \\ \hline
$\psi_B$ & $(n-1,1)$ & $a_1-1$ \\ \hline
$\psi_C$ & $(n-2,2)$ & $\frac12 (2 a_2 + a_1(a_1-3)) $ \\ \hline
$\psi_D$ & $(n-2,1,1)$ & $\frac12(-2 a_2 + a_1^2 -3 a_1 +2)$ \\ \hline
$\psi_E$ & $(n-3,3)$ & $\frac16(6a_3+6a_2(a_1-1) + a_1(a_1-1)(a_1-5))$ \\ \hline
$\psi_F$ & $(n-3,1,1,1)$ & $a_3-a_2(a_1-1)+{a_1-1 \choose 3} $\\ \hline
\end{tabular}
\endgroup

\medskip

Table 1: Low dimensional representations of $\sn$
\end{center}

We will prove that $\chi_\rho= \chi_F$ and so $\rho$ is equivalent to $\psi_F$.
First we will show that $\rho$ is irreducible.
Suppose on the contrary that $\rho$ is reducible, then it must hold that  
$$
\chi_{\rho}  =  a\chi_A + a'\chi_{A'} + b\chi_B + b'\chi_{B'} + c\chi_C + c'\chi_{C'} + 
d\chi_D + d'\chi_{D'} + e\chi_E + e'\chi_{E'},
$$
for some non negative integers $a, a',b ,\ldots , e'$.  Let $x_a= a+a',\, x_b=b+b',\ldots, x_e=e+e'$, then by 
Remark~\ref{rem:evenperm} it holds that for any even permutation $\mu$ we have 
$$
\chi_{\rho}(\mu) =  x_a\chi_A(\mu) + x_b\chi_B(\mu) + x_c\chi_C(\mu) + x_d\chi_D(\mu) + x_e\chi_E(\mu).
$$
We will show that this is not possible by evaluating the above expression for some specific even permutations $\mu$. The formulas below are applications of Proposition~\ref{prop:action}.

\begin{itemize}
	\item Consider the element $\mu=\overline \omega_1\overline \omega_2$, in this case $\mu=(1\, 2\,3)$ and
	\begin{equation}
\rho(\mu)( \alpha_{r,s,t} )= 
\begin{cases}
\alpha_{1,2,3} & \text{if $r=1$, $s=2$ and $t=3$}\\
\alpha_{1,2,t}^{-1} & \text{if $r=1$, $s=3$ and $t\geq4$}\\
\alpha_{1,3,t}^{-1} & \text{if $r=2$, $s=3$ and $t\geq4$}\\
\alpha_{\mu(r),\mu(s),\mu(t)} & \text{otherwise}.
\end{cases}
\end{equation}
So, we have $\binom{n-4}{3}+1$ fixed elements.

\item Consider the element $\mu=\overline \omega_1\overline \omega_2\overline \omega_4 \overline \omega_5$, in this case $\mu=(1\, 2\, 3)(4\, 5\, 6)$ and
	\begin{equation}
\rho(\mu)( \alpha_{r,s,t} )= 
\begin{cases}
\alpha_{1,2,3} & \text{if $r=1$, $s=2$ and $t=3$}\\
\alpha_{4,5,6} & \text{if $r=4$, $s=5$ and $t=6$}\\
\alpha_{1,2,\mu(t)}^{-1} & \text{if $r=1$, $s=3$ and $t\geq4$}\\
\alpha_{1,3,\mu(t)}^{-1} & \text{if $r=2$, $s=3$ and $t\geq4$}\\
\alpha_{\mu(r),4,5}^{-1} & \text{if $r\leq 3$, $s=4$ and $t=6$}\\
\alpha_{\mu(r),4,6}^{-1} & \text{if $r\leq 3$, $s=5$ and $t=6$}\\
\alpha_{4,5,t}^{-1} & \text{if $r=4$, $s=6$ and $t\geq7$}\\
\alpha_{4,6,t}^{-1} & \text{if $r=5$, $s=6$ and $t\geq7$}\\
\alpha_{\mu(r),\mu(s),\mu(t)} & \text{otherwise}.
\end{cases}
\end{equation}
So, we have $\binom{n-7}{3}+2$ fixed elements.

\item Consider the element $\mu =\overline \omega_1\overline \omega_2\overline \omega_4\overline \omega_5\overline \omega_7\overline \omega_8$, in this case $\mu=(1\, 2\, 3)(4\, 5\, 6)(7\,  8\, 9)$ and
	\begin{equation}
\rho(\mu) (\alpha_{r,s,t} )= 
\begin{cases}
\alpha_{1,2,3} & \text{if $r=1$, $s=2$ and $t=3$}\\
\alpha_{4,5,6} & \text{if $r=4$, $s=5$ and $t=6$}\\
\alpha_{7,8,9} & \text{if $r=7$, $s=8$ and $t=9$}\\
\alpha_{r,r+1\mu(t)}^{-1} & \text{if $r=1,4,7$, $s=r+2$ and $t\geq r+3$}\\
\alpha_{r-1,r+1,\mu(t)}^{-1} & \text{if $r=2,5,8$, $s=r+1$ and $t\geq r+3$}\\
\alpha_{\mu(r),s,s+1}^{-1} & \text{if $s=4,7$, $t=s+2$ and $r<s$}\\
\alpha_{\mu(r),s-1,s+1}^{-1} & \text{if $s=5,8$, $t=s+1$ and $r<s-1$}\\
\alpha_{\mu(r),\mu(s),\mu(t)} & \text{otherwise}.
\end{cases}
\end{equation}
So, we have $\binom{n-10}{3}+3$ fixed elements.

	\item Consider the element $\mu =\overline \omega_1\overline \omega_3$, in this case $\mu=(1\, 2)(3\, 4)$ and
	\begin{equation}
\rho(\mu)( \alpha_{r,s,t} )= 
\begin{cases}
\alpha_{1,2,\mu(t)}^{-1} & \text{if $r=1$, $s=2$ and $t\geq3$}\\
\alpha_{2,3,4}^{-1} & \text{if $r=1$, $s=3$ and $t=4$}\\
\alpha_{1,3,4}^{-1} & \text{if $r=2$, $s=3$ and $t=4$}\\
\alpha_{3,4,t}^{-1} & \text{if $r=3$, $s=4$ and $t\geq5$}\\\alpha_{\mu(r),\mu(s),\mu(t)} & \text{otherwise}.
\end{cases}
\end{equation}
So, we have $\binom{n-5}{3}$ fixed elements and $2(n-5)$ twisted elements (elements mapped to their inverse).
\end{itemize}
Hence, we conclude that the character $\chi_{\rho}$  in these specific elements is given by:
\begin{itemize}
	\item $\chi_{\rho}(Id)=\binom{n-1}{3}$ (the dimension of the representation).
	\item $\chi_{\rho}(\overline \omega_1\overline \omega_2)=\binom{n-4}{3}+1$.
	\item $\chi_{\rho}(\overline \omega_1\overline \omega_2\overline \omega_4\overline \omega_5)=\binom{n-7}{3}+2$.
	\item $\chi_{\rho}(\overline \omega_1\overline \omega_2\overline \omega_4\overline \omega_5\overline \omega_7\overline \omega_8)=\binom{n-10}{3}+3$.
	\item $\chi_{\rho}(\overline \omega_1\overline \omega_3)=\binom{n-5}{3}-2(n-5)$.
\end{itemize}

Now, using the information of Table~1 we obtain the following five equations:

\begin{enumerate}
	\item $Id$. In this case we have $a_1=n$, $a_2=a_3=0$.
	$$
	\begin{array}{rcl}
	\chi_{\rho}(Id) & = & x_a + x_b\cdot(n-1) + x_c\cdot\frac{n(n-3)}{2} + x_d\cdot\frac{n^2-3n+2}{2} + x_e\cdot\frac{n(n-1)(n-5)}{6}
	\end{array}
	$$
	
	\item $\mu_1=\overline \omega_1\overline \omega_2$. In this case we have $a_1=n-3$, $a_2=0$, $a_3=1$.
	$$
	\begin{array}{rcl}
	\chi_{\rho}(\mu_1) & = & x_a + x_b\cdot(n-4) + x_c\cdot\frac{(n-3)(n-6)}{2} + x_d\cdot\frac{(n-3)^2-3(n-3)+2}{2}\\[2mm]
	& & + x_e\cdot\frac{(n-3)(n-4)(n-8)+6}{6}
	\end{array}
	$$
	
	\item $\mu_2=\overline \omega_1\overline \omega_2 \overline \omega_4 \overline \omega_5$. In this case we have $a_1=n-6$, $a_2=0$, $a_3=2$.
	$$
	\begin{array}{rcl}
	\chi_{\rho}(\mu_2) & = & x_a + x_b\cdot(n-7) + x_c\cdot\frac{(n-6)(n-9)}{2} + x_d\cdot\frac{(n-6)^2-3(n-6)+2}{2}\\[2mm] 
	& & + x_e\cdot\frac{(n-6)(n-7)(n-11)+12}{6}
	\end{array}
	$$
	
	\item $\mu_3= \overline \omega_1\overline \omega_2\overline \omega_4\overline \omega_5\overline \omega_7\overline \omega_8$. In this case we have $a_1=n-9$, $a_2=0$, $a_3=3$.
	$$
	\begin{array}{rcl}
	\chi_{\rho}(\mu_3) & = & x_a + x_b\cdot(n-10) + x_c\cdot\frac{(n-9)(n-12)}{2} + x_d\cdot\frac{(n-9)^2-3(n-9)+2}{2}\\[2mm]
	& & + x_e\cdot\frac{(n-9)(n-10)(n-14)+18}{6}
	\end{array}
	$$
	
	\item $\mu_4=\overline \omega_1\overline \omega_3$. In this case we have $a_1=n-4$, $a_2=2$, $a_3=0$.
	$$
	\begin{array}{rcl}
	\chi_{\rho}(\mu_4) & = & x_a + x_b\cdot(n-5) + x_c\cdot\frac{(n-4)(n-7)+4}{2} + x_d\cdot\frac{(n-4)^2-3(n-4)-2}{2}\\[2mm]
	& & + x_e\cdot\frac{(n-4)(n-5)(n-9)+12(n-5)}{6}
	\end{array}
	$$
	
\end{enumerate}
 
Solving this system of equations we find that the unique solution is $x_a=5-n$, $x_b=1$, $x_c=5-n$, $x_d=n-5$ and $x_e=1$. Since we are considering $n\geq 13$ we get that  $x_a=5-n<0$ which is a contradiction. 
So, $\rho$ is either (equivalent to) the representation $\psi_F$ or the representation $\psi_{F'}= \sign \cdot \psi_F$. Now we use the element $\mu=\overline \omega_1$ in order to be able to decide which of the two options is the correct one. 
Considering the action of $\mu=\overline \omega_1=(1\,2)$ as given by Lemma~\ref{action:generators}, we obtain
\begin{equation}\label{value:w1}
\chi_{\rho}(\overline \omega_1)=\binom{n-3}{3}-(n-3). 
\end{equation}
Notice that $\chi_{\rho}(\overline \omega_1)= \chi_{F'}(1\, 2)$ if and only if $\binom{n-3}{3}-(n-3)=(n-3) - \binom{n-3}{3}$ if and only if $n=2,3,7$. 
Since we are considering the case $n\geq 13$ we conclude that $\rho$ is equivalent to $\psi_F$. So the Young diagram related to $\rho$ is $(n-3,1,1,1)$:
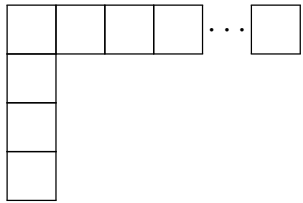
\begin{center}
\ytableausetup{centertableaux}
\begin{ytableau}
  \mbox{ } & \mbox{ } &  \mbox{ } & \mbox{ } &  \none[\cdots] & \mbox{ } \\
\mbox{ } &\none & \none & \none & \none & \none \\
\mbox{ }	&\none & \none & \none & \none & \none \\
\mbox{ }&\none & \none & \none & \none & \none 
\end{ytableau}
\end{center}

We still have to deal with the cases $n=6,7,\ldots,12$. We can follow the same idea as in the general case, but now using  the explicit character table of $\sn$, in each case instead of the information of Table~1. These character tables can e.g.\ be found using GAP (\cite{GAP}).

We still need to show that $\rho$ is the representation $\psi_F$. 
So, as in the general case, we first check that $\rho$ can not be written as a sum of lower dimensional representations, so that $\chi_\rho$ cannot be written as a sum of characters of lower degree. In order to achieve this, we assume that $\chi_\rho$ can be written as a combination (with non negative integer coefficients) of characters of lower degree and by evaluating the expression on some well chosen even permutations we look for a contradiction.

For example, when $n=8$ and we consider the permutation $\mu=\overline \omega_1 \overline \omega_3$ we find the equation
\[
\chi_{\rho}(\mu)=-5=1x_1+3x_2+4x_3+4x_4+2x_5+1x_6.
\]
which clearly does not have any solution in nonnegative integers $x_1$, $x_2$, \ldots, $x_6$, showing that 
$\rho$ cannot be decomposed into representations of lower degree.

For all the other cases (with $6\leq n \leq 12$) we also very quickly obtain the same result (sometimes using another permutation, e.g.\ $\overline \omega_1\overline \omega_2\overline \omega_4\overline \omega_5$).

Now, let us consider the case $n=8$ again. We already know that $\rho$ is either equivalent to the representation $\psi_F$ or to the representation $\psi_{F'}= \sign \cdot \psi_F$. Now, $\chi_\rho(\overline \omega_1)=5$ (see formula~\eqref{value:w1}), $\chi_F(  \overline \omega_1) =5 $ and $\chi_{F'}( \overline \omega_1) =-5$, allowing us to conclude that $\rho$ is equivalent to $\psi_F$.

The same argument works in all other cases (where for $n=7$ it holds that $\psi_F$ is equivalent to $\psi_{F'}$ and so this last step does not have to be done). 

\end{proof}

Now we recall a combinatorial procedure due to Stembridge (\cite{S}) allowing us to compute the eigenvalues appearing in the irreducible representations of the symmetric group.
 
Let $\psi\colon G\to {\rm GL}(k,\C)$ be a representation of a finite group and let $g\in G$ be an element of order $m$. Then the eigenvalues of $\psi(g)$ are $m$-th roots of unity, and so are of the form $\omega^{e_1}, \omega^{e_2},\ldots, \omega^{e_k}$, with $\omega=e^{2\pi i/m}$ a primitive $m$-th root of unity. Stembridge refers to the  intergers $e_j\ (mod\, m)$ as the \emph{cyclic exponents} of $g$ with respect to the representation $\psi$. Using these cyclic exponents we form a generation function in the variable $q$:
\[ P_{\psi, g}(q) = q^{e_1} + q^{e_2} + \cdots + q^{e_k} \]
which is well defined mod $1-q^m$.

In order to describe this generation function in full detail for representations of the symmetric group we need the notion of a standard tableau of shape $\lambda=(\lambda_1, \lambda_2, \ldots, \lambda_\ell)$, where 
$\lambda$ is a partition of $n$ as before. Informally, a standard tableau of shape $\lambda$ is a distribution of the integers $1,2,\ldots, n$ over the cells of the Young diagram $Y(\lambda)$ in such a way that the numbers increase both in rows and in columns. Formally, a standard tableau of shape $\lambda$ is a one-to-one map $T\colon Y(\lambda) \to \{1,2,\ldots,n\}$ with 
\begin{enumerate}
\item $T(i,j) < T(i, j+1)$ for all $1\leq i \leq \ell$ and $1 \leq j  < \lambda_i$ (increasing rows)
\item $T(i,j) < T(i+1, j)$ for all $1\leq i < \lambda_\ell$ and $1 \leq j \leq \lambda_{i+1}$ (increasing columns)
\end{enumerate}
As an example, the following picture shows a standard tableau of size $\lambda=(5,2,2,1)$:
\[ 
\begin{ytableau}
1 & 2 & 5 & 9 & 10 \\
3& 6  & \none & \none & \none \\
4 & 8 & \none & \none & \none \\
7 & \none & \none & \none & \none \\
\end{ytableau}
\] 

The set of all standard tableaux of shape $\lambda$ is denoted by  ${\cal F}^{\lambda}$.
Let $T\in{\cal F}^{\lambda}$.  If the number $k+1$ ($1\leq k < n$) appears in a row strictly below $k$ in $T$, then $k$ is called a \emph{descent} of $T$, and $D(T)$ denotes  the set of all descents in $T$. So, for the 
example above, we have that $D(T)=\{2,\,3,\,5,\,6\}$.

Let $\mu=(\mu_1,\mu_2, \ldots, \mu_r )$ be a partition of $n$, and let $m=\operatorname{lcm}(\mu_1,\mu_2,\ldots, \mu_r)$ (the least common multiple of the numbers in the partition). Then $m$ is the   order of any permutation in $\sn$ of cycle--type $\mu$. We define an $n$--tuple $b_{\mu}=(b_{\mu}(1),\ldots ,b_{\mu}(n))$ of positive integers as follows
$$
b_{\mu} = \largeleft \underbrace{\frac{m}{\mu_1},\, \frac{2m}{\mu_1},\, \cdots,\, 
\frac{\mu_1 m}{\mu_1}=m}_{\mbox{$\mu_1$ terms}}, \,\underbrace{\frac{m}{\mu_2},\, \frac{2m}{\mu_2},\, \cdots,\, \frac{\mu_2 m}{\mu_2}=m}_{\mbox{$\mu_2$ terms}},\, \cdots,\,
\underbrace{\frac{m}{\mu_r},\, \frac{2m}{\mu_r},\, \cdots,\, \frac{\mu_r m}{\mu_r}=m}_{\mbox{$\mu_r$ terms}}
 \largeright.
$$
For example, when $\mu=(4,3,2,1)$, then $m=12$ and $b_\mu=(3,6,9,12,4,8,12,6,12,12)$.

Now consider two partitions $\lambda$ and $\mu$ of $n$.
For any standard tableau $T\in {\cal F}^\lambda$, we define the \emph{$\mu$-index} of $T$  as
$$
\operatorname{ind}_{\mu}(T) = \sum_{k\in D(T)} b_{\mu}(k)\ (mod\, m).
$$

So for the example of the standard tableau $T$ above with $\lambda=(5,2,2,1)$ and for $\mu=(4,3,2,1)$ we have that 
$$
\operatorname{ind}_{\mu}(T) = b_\mu(2) + b_\mu(3) + b_\mu(5) + b_\mu(6) \ (mod\, 12) = 6 + 9 +4 + 8\ (mod\, 12) = 3.
$$

Now, we have all the necessary background  to provide a good description of the generating function $P_{\psi,g}(g)$ in case $\psi$ is an irreducible representation of $\sn$. Since such a $\psi$ is determined 
by the associated partition $\lambda$ and the eigenvalues of $\psi(g)$ (and so also the generating function) only depend on the cycle--type $\mu$ of $g$, 
we will write 
$P_{\lambda,\mu}(q)$ to denote the cyclic exponent generating function corresponding to any permutation $w$ of $\sn$ of cycle--type $\mu$ for the irreducible representation of  $\sn$ corresponding to $\lambda$. This generating function (mod $1-q^m$) is completely determined by the following result.

\begin{thm}[{\cite[Theorem~3.3]{S}}]\label{th:s}
We have 
$$
P_{\lambda,\mu}(q) = \sum_{T\in {\cal F}^{\lambda}} q^{\operatorname{ind}_{\mu}(T)}\quad (\textrm{mod } 1-q^m);
$$
i.e., the cyclic exponents of $w$ with respect to $\lambda$ are the $\mu$-indices of the standard tableaux of shape $\lambda$.
\end{thm}

\reth{s} will be useful to prove that for the representation $\phi_2\colon \aut{P_n} \to 
\aut{
\frac
{\Gamma_2(\barpn[n-1])}
{\Gamma_3(\barpn[n-1])}
}$, the automorphism $\phi_2(\alpha)$ has 1 as an eigenvalue for all $\alpha\in \aut{P_n}$.

\begin{rem}
Consider $\lambda= (n-3, 1,1,1)$  and so the corresponding Young diagram $Y(\lambda)$ is: 
\begin{center}
\ytableausetup{centertableaux}
\begin{ytableau}
  \mbox{ } & \mbox{ } &  \mbox{ } & \mbox{ } & \none[\cdots] & \mbox{ } \\
\mbox{ }  & \none & \none & \none& \none & \none \\
\mbox{ }	 & \none & \none & \none & \none & \none\\
\mbox{ } & \none & \none & \none & \none & \none
\end{ytableau}
\end{center}
For every $i,j,k$ with $1\leq i<j<k< n$ there exists a standard tableau $T\in {\cal F}^\lambda$ with $D(T)=\brak{i,j,k}$ and all $D(T)$ are of this form. Indeed, take $T$ the unique standard tableau given by 
\begin{center}\scriptsize
\ytableausetup{centertableaux,boxsize=2.8em}
\begin{ytableau}
  \mbox{\scriptsize1 } &  \ast &  \ast & \ast & \none[\cdots] &  \ast \\
i+1   & \none & \none & \none& \none & \none \\
j+1	 & \none & \none & \none & \none & \none\\
k+1 & \none & \none & \none & \none & \none
\end{ytableau}
\end{center}
\end{rem}

\begin{thm}\label{th:ev}
Let $n\geq 6$ and let $\phi_2\colon \aut{P_{n-1}} \to \aut{\frac{\Gamma_2(\overline P_{n-1})}{\Gamma_3(\overline P_{n-1})}}$ be the induced morphism. For any $\alpha \in \aut{P_{n-1}}$, we have that  $\phi_2(\alpha)$ has 1 as an eigenvalue. 
\end{thm}

\begin{proof}
Recall that $\phi_2$ factors through the representation $\rho\colon S_{n} \to {\rm GL}( {n-1 \choose 3} , \C)$.
We proved in Theorem~\ref{rho:known} that $\rho$ is equivalent to the irreducible representation of $\sn$ determined by the partition $\lambda=(n-3,1,1,1)$. So in order to prove the theorem, it suffices to prove that $\rho(\sigma)$ has eigenvalue 1 for all $\sigma \in \sn$. 

We shall consider four cases depending on the cycle--type $\mu$ of $\sigma$. 
Recall that the cycle--type $\mu$ of $\sigma$ is a  partition of $n$ of the form $\mu=(\mu_1, \mu_2,\ldots,\mu_r)$ where $\mu_1\geq \mu_2\geq \ldots\geq \mu_r\geq 1$. 
Let $m=\operatorname{lcm}(\mu_1, \mu_2,\ldots,\mu_r)$. 

\medskip

First we shall consider the case $\mu=(\mu_1, \mu_2,\ldots,\mu_r)$ such that $n>\mu_1\geq 3$. 
Hence $b_{\mu}=\left( \frac{m}{\mu_1}, \frac{2m}{\mu_1}, \cdots, \frac{(\mu_1-1)m}{\mu_1}, m,  \cdots \right)$. 
Consider the standard tableau 
$$\scriptsize
\ytableausetup{centertableaux, boxsize=3.2em}
\begin{ytableau}
  1 & 3 & 4 & \cdots & \mu_1-1 & \mu_1+2 & \cdots\\
  2 & \none & \none & \none & \none  & \none  & \none\\
	\mu_1 & \none & \none & \none & \none  & \none  & \none\\
	\mu_1+1 & \none & \none & \none & \none  & \none  & \none
\end{ytableau}
$$
So, $D(T)=\brak{1,\mu_1-1,\mu_1}$. Hence the $\mu$-index in this case is 
$$
\begin{array}{rcl}
\operatorname{ind}_{\mu}(T) & = &  b_{\mu}(1) + b_{\mu}(\mu_1-1) + b_{\mu}(\mu_1)\\
& = & \frac{m}{\mu_1} + \frac{(\mu_1-1)m}{\mu_1} + \frac{\mu_1m}{\mu_1} = 2m\\
& \equiv &  0\ (mod\, m).
\end{array}
$$
It follows from \reth{s}  that in this case $\rho(\sigma)$ has 1 as an eigenvalue.

\medskip

Let $\mu=(n)$, so $\sigma$ is  the full cycle. Then $m=n$, $\mu_1=n$ and $b_{\mu}=\left( 1,2,3,\ldots,n \right)$. 
Consider the standard tableau 
$$\scriptsize
\ytableausetup{centertableaux, boxsize=2.9em}
\begin{ytableau}
  1 & 2 & 3 & 5 & 6 & \cdots & n-2\\
  4 & \none & \none & \none & \none  & \none  & \none\\
	n-1 & \none & \none & \none & \none  & \none  & \none\\
	n & \none & \none & \none & \none  & \none  & \none
\end{ytableau}
$$
So, $D(T)=\brak{3,n-2,n-1}$. Hence the $\mu$-index in this case is 
$$
\begin{array}{rcl}
\operatorname{ind}_{\mu}(T) & = &  b_{\mu}(3) + b_{\mu}(n-2) + b_{\mu}(n-1)\\
& = & 3 + (n-2) + (n-1) = 2\mu_1=2n=2m\\
& \equiv&  0\ (mod\, m).
\end{array}
$$
\reth{s} implies that also in this case $\rho(\sigma)$ has 1 as an eigenvalue. 

\medskip

Now we shall consider the case $\mu=(2,1,\cdots,1)$, the case in which we have exactly one transposition.
So, $m=2$ and $b_{\mu}=( 1,2,2,2,\ldots, 2)$. 
Consider the standard tableau 
$$\scriptsize
\ytableausetup{centertableaux, boxsize=2.9em}
\begin{ytableau}
  1 & 2 & 6 & 7 & \cdots & n-1 & n\\
  3 & \none & \none & \none & \none  & \none  & \none\\
	4 & \none & \none & \none & \none  & \none  & \none\\
	5 & \none & \none & \none & \none  & \none  & \none
\end{ytableau}
$$
So, $D(T)=\brak{2,3,4}$. Hence the $\mu$-index in this case is 
$$
\begin{array}{rcl}
\operatorname{ind}_{\mu}(T) & = &  b_{\mu}(2) + b_{\mu}(3) + b_{\mu}(4)\\
& = & 2+2+2 = 6\\
& \equiv &  0\ (mod\, m).
\end{array}
$$
Again, \reth{s} allows us to conclude that also  in this case $\rho(\sigma)$ has 1 as an eigenvalue.

\medskip

Finally we shall consider the case $\mu=(2,2,\cdots,2,1,\ldots,1)$, the case in which we have $\ell\geq 2$ transpositions.
Hence $m=2$ and $b_{\mu}=( 1,2,1,2,\ldots,1,2,2,\ldots, 2)$ where the numbers $1,2$ are repeated $\ell$ times. 
Consider the standard tableau 
$$ \scriptsize
\ytableausetup{centertableaux, boxsize=2.9em}
\begin{ytableau}
  1 & 5 & 6 & 7 & \cdots & n-1 & n\\
  2 & \none & \none & \none & \none  & \none  & \none\\
	3 & \none & \none & \none & \none  & \none  & \none\\
	4 & \none & \none & \none & \none  & \none  & \none
\end{ytableau}
$$
So, $D(T)=\brak{1,2,3}$. Hence the $\mu$-index in this case is 
$$
\begin{array}{rcl}
\operatorname{ind}_{\mu}(T) & = & b_{\mu}(1) + b_{\mu}(2) + b_{\mu}(3)\\
& = & 1+2+1 = 4\\
& \equiv&  0\ (mod\, m).
\end{array}
$$
Again, by  \reth{s} we know that  $\rho(\sigma)$ has 1 as an eigenvalue.

Hence, given $n\geq 6$, for every element $\alpha \in \aut{P_{n-1}}$ it holds that $\phi_2(\alpha)$ has 1 as an eigenvalue.
\end{proof}

\begin{thm}
Let $n\geq 3$. The Artin pure braid group $P_{n}$ has the $R_{\infty}$ property.
\end{thm}

\begin{proof}
The cases $n=3$ and $n=4$ were already proven before, so let $n\geq 5$.

Let $\alpha\in \aut{P_n}$. By the previous theorem we now that $\phi_2(\alpha)$ acts 
on the free abelian group $\frac{\Gamma_2(\overline P_n)}{\Gamma_3(\overline P_n)}$ as a matrix with eigenvalue 1. This means that $R(\phi_2(\alpha))=\infty$, which implies that $R(\alpha)=\infty$ (Remark~\ref{technique}).
Since this holds for any automorphism of $P_n$, we have shown that $P_n$ has the $R_\infty$ property.
\end{proof}

\bigskip

\end{document}